\documentclass{amsart}
\usepackage{amsmath}
\usepackage{amssymb}
\usepackage{color}
\input xy
\xyoption{all}
\newtheorem{Lemma}{Lemma}[section]
\newtheorem{remark}[Lemma]{Remark}

\newtheorem{theorem}[Lemma]{Theorem}
\newtheorem{lemma}[Lemma]{Lemma}
\newtheorem{proposition}[Lemma]{Proposition}
\newtheorem{corollary}[Lemma]{Corollary}

\newtheorem{example}[Lemma]{Example}

\newenvironment{Proof}{{\sc Proof.}\ }{~$\square$\vspace{0.2truecm}}
\newcommand{\Cal}[1]{{\mathcal #1}}
\newcommand{\End}{\operatorname{End}}
\newcommand{\Tor}{\operatorname{Tor}}
\newcommand{\Hom}{\operatorname{Hom}}
\newcommand{\Gen}{\operatorname{Gen}}
\newcommand{\Ext}{\operatorname{Ext}}

\newcommand{\Mod}{\operatorname{Mod-\!}}

\newcommand{\lMod}{\mbox{\rm -Mod}}

\DeclareMathOperator{\wed}{w.d.}
\DeclareMathOperator{\pd}{p.d.}
\newcommand{\cmat}{\left(\begin{array}}
\newcommand{\fmat}{\end{array}\right)}

\DeclareMathOperator{\reg}{reg}

\DeclareMathOperator{\Fdim}{F.dim}

\begin{document}
   \title{Equivalence of Some Homological Conditions for Ring Epimorphisms}
  \author[Alberto Facchini]{Alberto Facchini}
\address{Dipartimento di Matematica, Universit\`a di Padova, 35121 Padova, Italy}
 \email{facchini@math.unipd.it}
\thanks{The first author was partially supported by Dipartimento di Matematica ``Tullio Levi-Civita'' of Universit\`a di Padova (Project BIRD163492/16 ``Categorical homological methods in the study of algebraic structures'' and Research program DOR1690814 ``Anelli e categorie di moduli''). The second author was supported by a grant from  IPM}
 \author[Zahra Nazemian]{Zahra Nazemian}
\address{School of Mathematics, Institute for Research in Fundamental Sciences (IPM), Tehran, Iran}
 \email{z\_nazemian@yahoo.com}
\keywords{Cotorsion pair, Flat epimorphism of rings, Strongly flat module. \\ \protect \indent 2010 {\it Mathematics Subject Classification.} Primary 13E30.
Secondary 16S85.} 

      \begin{abstract} Let $R$ be a right and left Ore ring, $S$ its set of regular elements and    $Q = R[S^{-1}] = [S^{-1}] R$ the classical ring of quotients of $R$.  We prove that if $\Fdim(Q_Q) = 0$, then 
the following conditions are equivalent:
 $(i)$~Flat right
$R$-modules are strongly flat.
 $ (ii)$ Matlis-cotorsion right $R$-modules are Enochs-cotorsion.
$(iii) $ $h$-divisible right $R$-modules are weak-injective.
 $(iv)$~Homomorphic images of weak-injective right $R$-modules are weak-injective.
  $(v)$ Homomorphic images of injective right $R$-modules are weak-injective.
$(vi)$~Right $R$-modules of weak dimension $ \le 1$ are of projective dimension $\le1$.
 $(vii)$~The cotorsion pairs $(\Cal P_1,\Cal D)$ and $(\Cal F_1,\Cal W\Cal I)$ coincide.
  $(viii)$ Divisible right $R$-modules are weak-injective. This extends a result by Fuchs and Salce (2017) for modules over a commutative ring $R$.
\end{abstract}

    \maketitle

\section{Introduction}

In  \cite{FS}, Fuchs and Salce proved the equivalence of nine conditions for modules over commutative rings $R$ with perfect ring of quotients $Q$. The aim of this paper is to show that the equivalence of seven of their conditions also holds for noncommutative right and left Ore rings $R$ for which  $\Fdim(Q_Q) = 0$. Here
  $Q = R[S^{-1}] = [S^{-1}] R$, where $S$ is the set of regular elements of $R$. Notice that a commutative ring $Q$ is perfect if and only if $\Fdim(Q_Q) = 0$ \cite[pp.~466--468]{Bass}. Thus this paper is a genuine extension of part of the results by Fuchs and Salce.

The history of this line of research begins in the theory of abelian groups. The term {\em cotorsion} first appears in Harrison \cite{36}, who defined  cotorsion abelian groups as those reduced groups $G$ for which every short exact sequence $0\to G\to A\to B\to0$ with $B$ torsion-free splits, that is, the reduced abelian groups $G$ for which $\Ext(B,G)=0$ for every torsion-free abelian group $B$. The theory of cotorsion abelian groups was then extended to modules over commutative rings $R$ by Matlis \cite{MatlisMAMS, MatlisTF, 22}. As Matlis says in \cite[Introduction, p.~3]{MatlisTF}: ``Without doubt
there are no ideas in the general theory of integral domains which are
more fundamental in nature than those of cotorsion modules and completions as well as the relations between them.''  Matlis' results were soon extended to the case of noncommutative rings $R$, for instance in the works by Sandomierski \cite{sando}, who obtained very elegant results in the case of rings $R$ with a semisimple maximal quotient ring $Q$. Cotorsion theory then received a great impulse with the proof of the so called "flat cover conjecture": every module has a flat cover. This had been conjectured by Enochs \cite{7}, and proved with two different solutions by Bican, El Bashir and Enochs \cite{4}. One proof is an application of a theorem of Eklof and Trlifaj \cite{6} that guarantees the existence of ``enough projectives and injectives'' for suitable cotorsion theories.

\medskip

Fuchs and Salce \cite[Theorem~7.1]{FS} proved that if $R$ is an order in a commutative perfect ring $Q$, then the following conditions are
equivalent:

$(i)$ $R$ is an almost perfect ring.

$(ii)$ 
Flat $R$-modules are strongly flat.

$(iii)$ Matlis-cotorsion $R$-modules are Enochs-cotorsion.

$(iv)$ $R$-modules of w.d.$\le 1$ are of p.d.$\le 1$.

$(v)$ The cotorsion pairs $(\Cal P_1,\Cal D)$ and ($\Cal F_1, \Cal W\Cal I)$ are equal.

$(vi)$ Divisible $R$-modules are weak-injective.

$(vii)$ $h$-divisible $R$-modules are weak-injective.

$(viii)$ Homomorphic images of weak-injective $R$-modules are weak-injective.

$(ix)$ $R$ is $h$-local and $Q/R$ is semi-artinian.

\medskip

Here $\Cal P_1$ and $\Cal F_1$ denote the classes of all $R$-modules of projective dimension $\le1$, of weak dimension $\le 1$ respectively.

\medskip

We prove that if $R$ is a right and left Ore ring and  $\Fdim(Q_Q) = 0$, where
  $Q = R[S^{-1}] = [S^{-1}] R$, then 
the following conditions are equivalent:
 
$(i)$ Flat right
$R$-modules are strongly flat.
 
$ (ii)$ Matlis-cotorsion right $R$-modules are Enochs-cotorsion.

$(iii) $ $h$-divisible right $R$-modules are weak-injective.
 
$(iv)$ Homomorphic images of weak-injective right $R$-modules are weak-injective.
  
$(v)$ Homomorphic images of injective right $R$-modules are weak-injective.

$(vi)$ Right $R$-modules of $\wed \le 1$ are of $\pd\le1$.
 
$(vii)$ The cotorsion pairs $(\Cal P_1,\Cal D)$ and $(\Cal F_1,\Cal W\Cal I)$ coincide.
 
 $(viii)$ Divisible right $R$-modules are weak-injective.

\medskip

Our results in this paper are organized in sections with an increasing number of hypotheses on the extension of rings $R\subseteq Q$. In Section~\ref{0}, we assume that the inclusion $R\to Q$ be an epimorphism in the category of rings and that $\Tor_1^R(Q,Q)=0$. In Section~\ref{1}, the assumption $\Tor_1^R(Q,Q)=0$ is replaced by the stronger assumption that ${}_RQ$ be a flat left $R$-module. In Section~\ref{2}, we consider the case where the Gabriel topology $\Cal F$ consisting of all right ideals $I$ of $R$ with $IQ=Q$ has a basis of principal right ideals. In Section~\ref{4}, we consider the case of rings $R$ for which the set $S$ of all their regular elements is both a right denominator set and a left denominator set (right and left Ore ring), and we assume that $Q=R[S^{-1}]=[S^{-1}]R$ (the classical right and left ring of  quotients of $R$) and that $\Fdim(Q_Q)=0$. Under these hypotheses, we prove the equivalence of seven of the nine conditions considered by Fuchs and Salce in \cite{FS}.

\smallskip

As far as notation and terminology are concerned, for a ring $Q$, $\Fdim(Q_Q) = 0$ means that every right $R$-module has projective dimension $0$ or $\infty$. For a ring $Q$, $\Fdim(Q_Q) = 0$ if and only if $R$ is right perfect and every simple right $R$-module is a homomorphic image
of an injective module \cite[Theorem~6.3]{Bass}. For a commutative ring $Q$, $\Fdim(Q) = 0$ if and only if $R$ is perfect. A right and left Ore ring is a ring $R$ such that, for all elements $x,y\in R$ with $x$ regular, there exist elements $u,v,u,v'$ with $v$ and $v'$ regular, $ux=vy$ and $xu'=yv'$. If $S$ denotes the set of all regular elements of $R$, the condition ``$R$ is a right and left Ore ring'' is equivalent to the existence of both classical rings of quotients $R[S^{-1}]$ and $[S^{-1}]R$. In this case, they necessarily coincide.

  \section{Bimorphisms $R\to Q$ and the condition $\Tor_1^R(Q,Q)=0$}\label{0}

{\em In this section, $R$ and $Q$ are rings, $\varphi \colon R\to Q$ is a bimorphism in the category of rings, that is, $\varphi$ is both a monomorphism and an epimorphism, and $\Tor_1^R(Q,Q)=0$.} Set $K:=Q/\varphi(R)$. Then the pair $(Q,\varphi)$ has the following properties:
\begin{itemize}
\item[(1)] The mapping $\varphi$ is injective, and is a ring morphism, so that $R$ can be viewed as a subring of $Q$ via $\varphi$. We will always identify via $\varphi$ the isomorphic rings $R$ and $\varphi(R)$, so that $\varphi$ will be always seen as an inclusion.
\item[(2)] $\varphi$ is an epimorphism in the category of associative rings, that is, if, for every pair of morphisms of rings $\psi,\omega\colon Q\to Q'$, $\psi\varphi=\omega\varphi$ implies $\psi=\omega$.
\item[(3)] $\Hom(M_R,N_R)=\Hom(M_Q,N_Q)$ for every pair of right $Q$-modules $M_Q,N_Q$ \cite[Proposition~XI.1.2(d)]{Stenstrom}.
\item[(4)] $\Tor_1^R(M_R,{}_RN)\cong \Tor_1^Q(M_Q,{}_QN)$ for every right $Q$-module $M_Q$ and every left $Q$-module ${}_QN$ \cite[Theorem 4.8]{Schofield}.
\item[(5)]\label{(5)}  $\Ext^1_R(M_R,N_R)\cong\Ext^1_Q(M_Q,N_Q)$ for every pair of right $Q$-modules $M_Q,N_Q$ and $\Ext^1_R(_RM,{}_RN)\cong\Ext^1_Q(_QM,{}_QN)$ for every pair of left $Q$-modules $_QM,{} _QN$ \cite[Theorem 4.8]{Schofield}.
\item[(6)] The class of all right $R$-modules $M_R$ with $M\otimes_RQ=0$ is closed under homomorphic images, direct sums and extensions, and therefore it is the torsion class of a torsion theory for $\Mod R$. We will denote by $t(M_R)$ the torsion submodule of any right $R$-module $M_R$ in this torsion theory. In all the paper, whenever we say ``torsion'' or ``torsion-free'', we refer to this torsion theory.
\item[(7)]\label{(7)} A right $R$-module $M_R$ is a right $Q$-module $M_Q$ if and only if $$\Ext^1
_R(K_R,M_R) = 0\quad\mbox{\rm and}\quad \Hom(K_R,M_R) = 0$$ (\cite[Theorem 2.6]{AS} and \cite[Proposition 4.12]{18}).
As a consequence of this, if a right $R$-module $M_R$ is a right $Q$-module $M_Q$, then its unique right $Q$-module structure is given by the canonical isomorphism $\Hom(Q_R,M_R) \to M_R$ \cite[Remark~2.7]{AS}.
\item[(8)] The $R$-$R$-bimodule $Q\otimes_R Q$ is isomorphic to the $R$-$R$-bimodule $Q$ via the canonical isomorphism induced by the multiplication $\cdot\colon Q\times Q\to Q$ of the ring $Q$ \cite[Proposition~XI.1.2]{Stenstrom}.
\item[(9)] Every right $Q$-module is a torsion-free $R$-module.

[Proof of (9): Let $M_Q$ be a right $Q$-module. In order to prove that $M_R$ is torsion-free, we must prove that for every right $R$-module $T_R$, $T\otimes_RQ=0$ implies $\Hom(T_R, M_R)=0$. Since $M_R$ is a $Q$-module,  $\Hom(Q_R,M_R)\cong M_R$, as we have remarked in (7). Thus $$\Hom(T_R, M_R)\cong \Hom(T_R, \Hom(Q_R,M_R))\cong \Hom(T\otimes_RQ,M_R)=0.]$$
\item[(10)] $K\otimes_RQ=0$ and $\Tor_1^R(K_R, {}_RQ)$=0. 

[Proof of (10): Consider the short exact sequence of $R$-$R$-bimodules \begin{equation*} 0\rightarrow
R\rightarrow Q\rightarrow K\rightarrow 0\end{equation*} and tensor it with the left module $_RQ$, getting the short exact sequence of abelian groups  $0\rightarrow \Tor_1^R(K_R, {}_RQ)\rightarrow Q\rightarrow Q\otimes_R Q\rightarrow K\otimes_R Q\rightarrow 0$. The central mapping $Q_R\rightarrow Q\otimes_R Q$, $q\mapsto 1\otimes q$, is clearly a right inverse of the isomorphism $Q\otimes_R Q\to Q$ considered in (8). Since any bijection has a unique right inverse, which is also the left inverse and is a bijection, it follows that the mapping $Q_R\rightarrow Q\otimes_R Q$ is a bijection. Hence its kernel $\Tor_1^R(K_R, {}_RQ)$ and cokernel $Q\otimes_R K$ are both zero.]
\end{itemize}

\begin{remark}\label{VHOU}{\rm (a) Notice that our conditions for this section on the extension $R\subseteq Q$ are left/right symmetric. Therefore all the definions we  give and all the results we  prove in this section about right modules are always true, mutatis mutandis, for left $R$-modules as well.

(b) The ring epimorphisms $\varphi \colon R\to Q$ such that $\Tor^R_n(Q,Q)=0$ for all $n\ge 1$ are sometimes called {\em
homological ring epimorphisms} (\cite[Definition~2.3]{AS}, \cite{18}).}\end{remark}

\begin{Lemma}\label{ppp} The following conditions are equivalent for a right $R$-module $N_R$:

{\rm (1)} Every homomorphism $R_R\to N_R$ extends to a right $R$-module morphism $Q\to N_R$.

{\rm (2)} $N_R$ is a homomorphic image of a right $Q$-module.

{\rm (3)} $N_R$ is a homomorphic image of a direct sum of copies of $Q$.\end{Lemma}

\begin{Proof} (1)${}\Rightarrow{}$(3) The module $N_R$ is a homomorphic image of a direct sum of copies of $R_R$, so that there is an epimorphism $\pi\colon{}R^{(X)}_R\to N_R$. Each restriction $\pi_x\colon R_R\to N_R$ extends to a right $R$-module morphism $Q_R\to N_R$ by (1).

(3)${}\Rightarrow{}$(2) is trivial.

(2)${}\Rightarrow{}$(1) If $N_R$ is a homomorphic image of a right $Q$-module, $N_R$ is a homomorphic image of a free right $Q$-module, $N_R\cong Q_R^{(X)}/S$ say. Every homomorphism $R_R\to Q_R^{(X)}/S$ is left multiplication by an element $(q_x)_{x\in X}+S$ of $Q_R^{(X)}/S$. Left multiplication by the element $(q_x)_{x\in X}$ is a right $Q$-module morphism $Q_Q\to Q^{(X)}_R$, which composed with the canonical projection $Q^{(X)}_R\to {}Q^{(X)}/S\cong N_R$ yields the suitable extension $Q\to N_R$.
\end{Proof}

We say that a right $R$-module is {\em $h$-divisible} if it satisfies the equivalent conditions of Lemma~\ref{ppp}. Clearly, any direct sum of $h$-divisible right $R$-modules is $h$-divisible, homomorphic images of injective modules are $h$-divisible, and any right $R$-module $B_R$ contains a unique largest $h$-divisible submodule $h(B_R)$ that contains every $h$-divisible submodule of $B_R$. We will say that $B_R$ is {\em $h$-reduced} if $h(B_R)=0$ (equivalently, if $B_R$ has no nonzero $h$-divisible submodule, equivalently if $\Hom(Q_R,B_R)=0$).
 
 \begin{Lemma} {\rm \cite[Lemma 3.3]{AS}} For every right $R$-module $M_R$, the image of the canonical morphism $\Hom(Q_R,M_R)\to M$ is the unique largest $h$-divisible submodule $h(M_R)$ of $M_R$.\end{Lemma}

\begin{proposition} \label{factor} {\rm \cite[Lemma 3.4]{AS}}. The following conditions are equivalent:

{\rm (1)}  $h(M_R/h(M_R))=0$ for every right $R$-module $M_R$.

{\rm (2)} For every short exact sequence $ 0  \to A_R \to B_R \to C_R \to 0$ of right 
$R$-modules, if $A_R$ and $C_R$ are $h$-divisible, then $B_R$ is also $h$-divisible.\end{proposition}

Several of our results can be stated in the following terminology and notation, as in \cite[\S~8]{andersonfuller}. Let~$\mathcal{U}$ denote a class of right $R$-modules and $\Gen(\mathcal{U})$ the class of all the right modules $M$ {\em generated by} $\Cal U$, that is, 
for which there exist an  indexed set  $ (U_\alpha )_ {\alpha \in A}$
in   $\mathcal{U}$ and an epimorphism $ \oplus_{\alpha \in A} U_ {\alpha} \to  M$.
For any right module $M$, set Tr$_M (\mathcal{U}) := \sum \{\,f(U) \mid  f\colon U \to M$
 is a homomorphism for some $ U \in \mathcal{U}\, \}$.  Thus $M \in{} \Gen(\mathcal{U})$ if and only if Tr$_M (\mathcal{U}) = M$. If $\mathcal{U}$ consists of a unique module $U$, we will write $\Gen({U})$ and Tr$_M ({U})$.
Thus ``$h$-divisible'' means ``generated by $Q_R$'', and we have that Tr$_M (Q_R)=h(M)$ for every right $R$-module $M$.

\begin{Lemma}\label{pjb1}
Every right $R$-module generated by $K_R$ is torsion and $h$-divisible.\end{Lemma}

\begin{Proof} The right $R$-module $K_R$ is clearly $h$-divisible, and is torsion by (10). Both the classes of $h$-divisible modules and torsion modules are closed under direct sums and homomorphic images. Thus every module generated $K_R$ is torsion and $h$-divisible.\end{Proof}

\begin{proposition} The right $R$-module $\Hom(K_R,M_R)$ is torsion-free for every right 
$R$-module $M_R$.\end{proposition}

\begin{Proof} Apply the functor $\Hom(-,M_R)$ to the short exact sequence \begin{equation*} 0\rightarrow
R\rightarrow Q\rightarrow K\rightarrow 0\end{equation*} of $R$-$R$-bimodules, getting an exact sequence $$0\to\Hom(K_R,M_R)\to\Hom(Q_R,M_R)\to\Hom(R_R,M_R)$$ of 
right $R$-modules. Now $\Hom(Q_R,M_R)$ is a right $Q$-module, hence it is a torsion-free right $R$-module by (9). In any 
torsion theory, submodules of torsion-free modules are torsion-free. Thus the submodule 
$\Hom(K_R,M_R)$ of the torsion-free right $R$-module $\Hom(Q_R,M_R)$ is torsion-free.\end{Proof}

A right module $M_R$ is {\em Matlis-cotorsion} if $\Ext^1_R(Q_R,M_R)=0$.
By (5), all right $Q$-modules are 
Matlis-cotorsion right $R$-modules. More generally, let $ \mathcal {A}$ be a  class of right $R$-modules. Set $ ^ \bot \mathcal {A}:= \{ \,B\in \Mod R \mid \Ext ^1 (B, A) = 0$ for every $A \in \mathcal {A}\,\}$.
Similarly, $\mathcal {A}^ \bot := \{ \,B\in \Mod R \mid\Ext ^1 (A, B) = 0$ for every $A \in \mathcal {A}\,\}$. Note that
$\mathcal {A} \subseteq{} ^ \bot (\mathcal {A}^ \bot)$ and $\mathcal {A} \subseteq ( ^\bot \mathcal {A} ) ^ \bot$  always.

If the class of $\mathcal {A}$
consists of a single element, $A$ say, we will simply write $^ \bot A$ and $A ^ \bot$. Thus the class of Matlis-cotorsion modules is the class $Q^ \bot$.

\begin{theorem} \label{1.5} Let $ 0  \to A_R \to B_R \to C_R \to 0$ be a short exact sequence of right 
$R$-modules. Then:

{\rm (1)} If $A_R$ and $C_R$ are $h$-reduced, then $B_R$ is also $h$-reduced.

{\rm (2)} If $A_R$ and $C_R$ are Matlis-cotorsion, then $B_R$ is also Matlis-cotorsion.

{\rm (3)} If $B_R$ is Matlis-cotorsion $h$-reduced, then $A_R$ is Matlis-cotorsion if and only if $C_R$ is $h$-reduced.\end{theorem}

\begin{Proof} Apply the functor $\Hom(Q_R,-)\colon\Mod R\to{}$Ab to the given short exact sequence, getting the corresponding long exact sequence.\end{Proof}

\begin{theorem} \label{1.6} The right $R$-module $\Hom({}_RK_R,M_R)$ is Matlis-cotorsion $h$-reduced for every right $R$-module $M_R$.\end{theorem}

\begin{Proof} We know that there is a canonical isomorphism $$\Hom(Q_R,\Hom({}_RK_R,M_R))\cong\Hom(Q\otimes_RK,M_R).$$ Since the hypothesis $\Tor_1^R(Q,Q)$ in this section is left/right symmetric, from (10) we know that $Q\otimes_RK=0$, so that $\Hom({}_RK_R,M_R)$ is $h$-reduced for every right $R$-module $M_R$. Now let $E_R$ be an injective right $R$-module containing $M_R$ and consider the exact sequence \begin{equation}0\to \Hom({}_RK_R,M_R)\to \Hom({}_RK_R,E_R)\to \Hom({}_RK_R,E/M).\label{VH}\end{equation} All the three right $R$-modules in this exact sequence are $h$-reduced by the first part of this proof. 

Now $\Tor^R_1(Q_R,{}_RK)=0$ by (10).
For the module $\Hom({}_RK_R,E_R)$, we have that $\Ext^1_R(Q_R,\Hom({}_RK_R,E_R))\cong \Hom_R(\Tor_1^R(Q_R,{}_RK),E_R)$ \cite[Proposition~VI.5.1]{CartanEilenberg}. It follows that $\Hom({}_RK_R,E_R)$ is Matlis-cotorsion. Now apply Theorem~\ref{1.5}(3) to the exact sequence (\ref{VH}), getting that $\Hom({}_RK_R,M_R)$ is Matlis-cotorsion as well.
\end{Proof}

A {\em cotorsion pair} is a pair $\mathfrak C=(\Cal A, \Cal B)$ of classes of right modules over
the ring $R$ such that $\Cal A ={}^ \bot \Cal B$ and $\Cal B = \Cal A^ \bot$. 
The class $\Cal A$ is always closed under arbitrary direct sums and contains
all projective right $R$-modules. Dually, the class $\Cal B$ is closed under direct products and contains all injective right $R$-modules. 

\section{Left flat bimorphisms}\label{1}

{\em In this section, $R$ and $Q$ are rings, $\varphi \colon R\to Q$ is a bimorphism in the category of rings, that is, $\varphi$ is both a monomorphism and an epimorphism, and ${}_RQ$ is a flat left $R$-module.} 
For examples of bimorphisms $R\subseteq Q$ with $\Tor^R_1(Q,Q)=0$ but $Q_R$ not flat, that is, bimorphisms satisfying the hypotheses of Section~\ref{0} but not those of Section~\ref{1}, see 
\cite[Examples 3.11(1) and 4.17]{AS} and Example~\ref{right, but not left}. Note that now, in this section, the conditions on the extension $R\subseteq Q$ are not left/right symmetric anymore, so that we must now distinguish between the behaviour of right modules and that of left modules (cf.~Remark~\ref{VHOU}).

The pair $(Q,\varphi)$ has the following properties:
\begin{itemize}
\item[(11)] The inclusion of $R$ into its maximal right ring of quotients $Q_{\max}(R)$ factors through the mapping $\varphi$ \cite[Theorem~XI.4.1]{Stenstrom}.
\item[(12)] $\varphi\colon R\to Q$ is the canonical homomorphism of $R$ into its right localization $R_{\Cal F}$, where $\Cal F=\{\,I\mid I$ is a right ideal of $R$ and $\varphi(I)Q=Q\,\}$ is a Gabriel topology consisting of dense right ideals. Moreover, $\Cal F$ has a basis consisting of finitely generated right ideals \cite[Theorem~XI.2.1 and Proposition~XI.3.4]{Stenstrom}.
\item[(13)] Every right $Q$-module is isomorphic to $M_{\Cal F}\cong M\otimes_R Q$ for some right $R$-module $M_R$ \cite[Proposition~XI.3.4]{Stenstrom}.
\item[(14)] The full subcategory of $\Mod R$ whose objects are all $\Cal F$-closed $R$-modules (that is, the modules $M_R$ such that, for every $I\in\Cal F$, every right $R$-module morphism $I\to M_R$ extends to a morphism $R_R\to M_R$ in a unique way) is equivalent to the category $\Mod Q$ \cite[Proposition~XI.3.4(a)]{Stenstrom}.
\item[(15)]\label{(15)} For every right $R$-module $M_R$, the kernel of the canonical right $R$-module morphism $M_R\to M\otimes_R Q$ is the torsion submodule $t(M_R)$ of $M_R$ \cite[Proposition~XI.3.4(f)]{Stenstrom}.
\item[(16)]\label{(16)} 
The torsion submodule $t(M_R)$ of any right $R$-module $M_R$ is isomorphic to $\Tor_1^R(M_R, {}_RK)$ \cite[Proposition~XI.1.2(e)]{Stenstrom}.

[Proof of (16): Consider the short exact sequence of $R$-$R$-bimodules $0\rightarrow
R\rightarrow Q\rightarrow K\rightarrow 0$ and tensor it with the right module $M_R$, getting  the short exact sequence of right $R$-modules  $0\rightarrow \Tor_1^R(M_R, {}_RK)\rightarrow M_R\rightarrow M_R\otimes_R Q\rightarrow M_R\otimes_R K\rightarrow 0$. Now (16) follows from (15).]

\item[(17)] Every right ideal of $Q$ is extended from a right ideal of $R$, that is, $I=\varphi^{-1}(I)Q$ for every right ideal $I$ of $Q$ \cite[Proposition~4(ii)]{Knight}.
\item[(18)] $\Ext^n_R(M_R,N_R)\cong\Ext^n_Q(M\otimes_R{}Q,N)$ for every $n\ge 0$, every right $R$-module $M_R$ and every right $Q$-module $N_Q$ (\cite[Page 232]{Stenstrom} or \cite[Page 118]{CartanEilenberg}).
\end{itemize}

It is possible to prove that, for any ring $R$, there exists a {\em maximal} left flat bimorphism $\overline{\varphi} \colon R\to \overline{Q}$, satisfying the following universal property: for any left flat bimorphism  $\varphi \colon R\to Q$, there exists a unique ring morphism $\beta\colon Q\to\overline{Q}$ such that $\beta\varphi=\overline{\varphi}$. Since the maximal left flat bimorphism is the solution of a universal property, $\overline{Q}$ is  unique up to isomorphism, in the following sense: if $\overline{\varphi} \colon R\to \overline{Q}$ and $\overline{\varphi_0} \colon R\to \overline{Q_0}$ are any two maximal left flat bimorphism, there exists a unique ring isomorphism $\beta\colon \overline{Q}\to\overline{Q_0}$ such that $\beta\varphi=\overline{\varphi_0}$.

\bigskip

For any ring $R$, we can consider the set $\Cal L_R$ of all subrings $Q$ of the maximal ring of quotients $Q_{\max}(R)$ such that the inclusion $R\to Q$ is a bimorphism and $_RQ$ is flat, and we can partially order $\Cal L_R$ by set inclusion.  Then  $\Cal L_R$ is a bounded complete lattice, where (1) the least element of $\Cal L_R$ is $R$, (2) the least upper bound of two elements $Q,Q'\in \Cal L_R$ is the ring generated by $Q$ and $Q'$, that is, the set of all finite sums of products of the form $q_1q'_1\dots q_nq'_n$, with $q_1,\dots,q_n\in Q$, $q'_1,\dots,q'_n\in Q'$, (3) the greatest lower bound of two elements $Q,Q'\in \Cal L_R$ is the union of all the subrings in $\Cal L_R$ contained in $Q\cap Q'$, and (4) the greatest element of $\Cal L_R$ is the ring $\overline{Q}$, where $\overline{Q}$ is the subring of $Q_{\max}(R)$ corresponding to the maximal left flat bimorphism $\overline{\varphi} \colon R\to \overline{Q}$ considered in the previous paragraph \cite[proof of Theorem~XI.4.1]{Stenstrom}.

\bigskip

Recall that a left $R$-module $_RD$ is {\em divisible} if $D=ID$ for every $I\in\Cal F$ (equivalently, if $M\otimes_RD=0$ for every torsion right $R$-module $M_R$ \cite[Proposition~VI.9.1]{Stenstrom}). $h$-divisible left $R$-modules are divisible.

\begin{theorem}\label{2.2} {\rm (1)} For every right $R$-module $M_R$, there is a short exact sequence of right $R$-modules \begin{equation}
\xymatrix{
0 \ar[r] & M_R/t(M_R) \ar[r] & M\otimes_R Q \ar[r] & M\otimes_R K \ar[r] & 0.
}\label{a}\tag{a}\end{equation}

{\rm (2)} For every left $R$-module $_RB$, there are two short exact sequences of left $R$-modules \begin{equation}
\xymatrix{
0 \ar[r] & \Hom({}_RK_R, {}_RB) \ar[r] & \Hom({}_RQ_R, {}_RB)  \ar[r] & h(_RB) \ar[r] & 0
}\label{b}\tag{b}\end{equation}
and \begin{equation}
\xymatrix{
0 \ar[r] & {}_RB/h(_RB) \ar[r] & \Ext^1_R({}_RK_R,{}_RB) \ar[r] & \Ext^1_R({}_RQ_R,{}_RB)  \ar[r] & 0.
}\label{c}\tag{c}\end{equation}\end{theorem}

\begin{Proof} Consider the short exact sequence of $R$-$R$-bimodules \begin{equation} 0\rightarrow
R\rightarrow Q\rightarrow K\rightarrow 0\label{d}\tag{d}\end{equation} and tensor it with the right module $M_R$, getting the short exact sequence of right $R$-modules  $0\rightarrow \Tor_1^R(M_R, {}_RK)\rightarrow M_R\rightarrow M_R\otimes_R Q\rightarrow M_R\otimes_R K\rightarrow 0$. This and properties (15) and (16) give short exact sequence (\ref{a}).

If we apply the contravariant functor $\Hom(-,{}_RB)\colon R\lMod\to R\lMod$ to exact sequence (\ref{d}), we obtain the exact sequence of left $R$-modules $$\begin{array}{l}0 \rightarrow \Hom({}_RK_R, {}_RB) \rightarrow \Hom({}_RQ_R, {}_RB) \stackrel{\beta}{\longrightarrow} {}_RB  \rightarrow \\ \qquad\qquad\qquad\rightarrow\Ext^1_R({}_RK_R,{}_RB)  \rightarrow \Ext^1_R({}_RQ_R,{}_RB)  \rightarrow 0,\end{array}
$$ where $\beta$ is defined by $\beta(f)=f(1)$ for every $f\in  \Hom({}_RQ, {}_RB)$. Now the image of $\beta$ is clearly $h({}_RB)$, and from this we get the two short exact sequences of left $R$-modules in (2). \end{Proof}

The next corollary shows that the class of torsion $h$-divisible right $R$-modules is generated by the right $R$-module $K_R$. More generally, Tr$_M ({K})=h(t(M_R))$ for any right $R$-module $M_R$.

{\begin{corollary}\label{pjb}
A right $R$-module is torsion $h$-divisible  if and only if it is generated by $K_R$. In particular, the right $R$-module $M\otimes_R{}K_R$ is a torsion $h$-divisible module for every right $R$-module $M_R$.\end{corollary}

\begin{Proof} 
Right $R$-modules generated by $K$ are torsion $h$-divisible by Lemma~\ref{pjb1}.

Conversely, assume that $M$ is a torsion $h$-divisible module. To see that $M$ is generated by $K$, we must show that,
for every $h\in M$, there exists $n \geq 1$ and $f \colon K^n \to M$ with $h \in f(K^n)$. Fix an element $h\in M$. 
By Lemma~\ref{ppp}(1), there exists $g \colon Q \to M$ such that $g (1) = h$.
 Set $S := \ker(g)$. Since $M$ is torsion, we have that $Q/S$ is torsion.
So $S\otimes Q = Q\otimes Q$, hence there exists $n \geq 1$ such that
 $1 = \sum_{i= 1} ^n s_i q_i $, where $s_i \in S$  and $q_i \in Q$.
Define a map $\varphi \colon K ^{n} \to Q/S$ setting, for all $t_1, \cdots, t_n \in Q$,
$\varphi(t_1+R, \cdots, t_n+R) = \sum _{i= 1}^n s_i t_i+S$.
If all the elements $t_i$ belong to $R$, then  $\sum _{i = 1}^n s_it_i \in S$, so this map $\varphi$ is well defined and is an $R$-module homomorphism.  The composite mapping of $\varphi \colon K ^{n} \to Q/S$  and the monomorphism $Q/S\to M$ induced by $g$ is the required mapping $f\colon K^n\to M$ whose image contains~$h$.

For the last part of the statement, apply the functor $-\otimes_R{}K_R\colon\Mod R\to\Mod R$ to an 
epimorphism $R_R^{(X)}\to M_R$, where $R_R^{(X)}$ is a suitable free right $R$-module.
\end{Proof}

As a consequence of Corollary \ref{pjb}, we have the following.

\begin{corollary}\label{Zahra} 
For every torsion right $R$-module $M_R$, the canonical mapping $$\pi\colon \Hom({}_RK_R, M_R)\otimes_RK \to h(M_R)$$ defined by $\pi(f\otimes x)=f(x)$ for every $f\in \Hom(K_R,M_R)$ and $x\in K$ is a right $R$-module epimorphism.\end{corollary}

\begin{Proof} The right $R$-module $\Hom({}_RK_R, M_R)\otimes_RK$ is a homomorphic image of the right $R$-module $\Hom({}_RK_R, M_R)\otimes_RQ$, which is a right $Q$-module. Thus the image of the canonical right $R$-module morphism $\Hom({}_RK_R, M_R)\otimes_RK \to M_R$ is contained in $h(M_R)$.
Conversely, note that $h(M) $ is torsion and $h$-divisible, so that $h(M) $ is generated by $K_R$ by Corollary~\ref{pjb}. Thus, if $x\in h(M)$, then there exists $n \geq 1$ such that
 $x$ belongs to the image of a morphism $f\colon K_R^n\to M_R $. Let $y = (k_1, \cdots, k_n)$ be such that $x = f(y)$. 
 For each $i=1,\dots,n$, let $\iota _i \colon K \to K^n$ be the canonical map. 
 Then  $\pi (\sum f \iota _i \otimes k_i ) = x$.  Thus  $\pi$ is an  epimorphism.
\end{Proof}

{\em Until the end of this section, we will consider {\em left} $R$-modules.} 

Define the class of {\em Matlis-cotorsion} left $R$-modules by $_R\mathcal {MC}:={} _RQ ^ \bot$ and
 the class of {\em strongly flat} left $R$-modules by $_R\mathcal {SF}:= {}^\bot(_R\mathcal {MC})$.
A left module $_RM$ will be said to be {\em Enochs-cotorsion} if
$\Ext_R^1 (_RF, {}_RM) = 0$ for all flat left $R$-modules $_RF$. Their class will be denoted
by $_R\mathcal{EC}$.
If $_R\mathcal {F}$ is the class of flat left $R$-modules, then
$(_R\mathcal {F}, _R\mathcal{EC})$ is a cotorsion pair \cite [ Lemma 7.1.4 ]{Jenda}. 
 Since $_RQ$ is flat, $Q^ \bot \supseteq \Cal F ^\bot = {} _R\mathcal{EC}$ and since $^\bot (_R\mathcal{EC}) = \mathcal{F}$, 
we have that strongly flat modules are flat.
Notice that the concept of Enochs-cotorsion left $R$-module is an ``absolute concept'', in the sense that it depends only on the ring $R$, while the concept of Matlis-cotorsion left $R$-module is a ``relative concept'', in the sense that it also depends on the choice of the overring $Q$ of $R$ with $_RQ$ flat.

A class $\Cal C$ of left $R$-modules is {\em precovering} if, for each left module $_RM$,
there exists a morphism $f\in \Hom_R(_RC,\, {}_RM)$ with $C\in\Cal C$ such that each morphism
$f_0\in\Hom(_RC_0,{}_RM)$ with $C_0\in\Cal C$ factors through $f$. Such an $f$ is called a {\em $\Cal C$-precover} of
$_RM$.

A precovering class $\Cal C$ of modules is called {\em special precovering} if every left $R$-module $_RM$ has a $\Cal C$-precover $f\colon C\to M$ which is an epimorphism and with 
$\ker(f)\in\Cal C^\bot$. Moreover, $\Cal C$ is called a {\em covering class} if every left  $R$-module $M$ has a $\Cal C$-precover $f \colon C \to M$ with the property that for every endomorphism $g$ of $C$ with $fg = f$, the endomorphism $g$ is necessarily an automorphism
of $C$. Such a $\Cal C$-precover $f$ is then called a
$\Cal C${\em -cover} of $_RM$. Dually, we define {\em preenveloping, special preenveloping,} and
{\em enveloping} classes of modules.
A cotorsion pair $\mathfrak C=(\Cal A, \Cal B)$ is {\em complete} if $\Cal A$ is a special precovering class (equivalently, if $\Cal B$ is a special preenveloping class \cite{Sa}). For instance, every
cotorsion pair generated by a set of modules is complete.

Note that, by
\cite[Theorem~6.11] {approx},  $(_R\mathcal{SF} ,{}_R\mathcal{MC})$ is a
complete cotorsion pair. Thus  every left $R$-module has a special
${}_R\mathcal{MC}$-preenvelope and  every left $R$-module has a special ${}_R\mathcal{SF}$-precover.

Now recall that  a left $R$-module $_RG$ is said to be {\em $\{Q\}$-filtered} if there exists an ordinal $\rho$ such that
$_RG$ is the union of a well-ordered ascending chain $\{\,G_\sigma \mid \sigma < \rho\,\}$
of submodules with $G_0=0$, $G_{\sigma+1}/G_\sigma\cong Q$
for every ordinal $\sigma<\rho$ and $G_\sigma=\bigcup_{\gamma<\sigma}G_\gamma$ for every limit ordinal $\sigma\le\rho$.
By \cite[Corollary  6.13]{approx}, the class $_R\mathcal{SF} $ consists of
all summands of
modules $_RN$ such that $_RN$ fits into an exact sequence of the form
$$ 0  \to {}_RF \to {}_RN \to {}_RG \to 0$$
where ${}_RF$ is free and ${}_RG$ is $\{Q\}$-filtered.

From (5), we get that:

\begin{lemma}\label{ExSten}
If $M$ and $N$ are projective left $Q$-modules, then $\Ext^1_R(M, N) = 0$.
\end{lemma}

Lemma \ref{ExSten} implies that a $Q$-filtered left $R$-module is a free $Q$-module (see, for example, 
  \cite [paragraph before the statement of Lemma~6.15] {approx}).  
Thus $_RG$ is a free $Q$-module. 
Therefore the class $_R\mathcal{SF}$ consists of
all summands of
modules $_RN$ such that $_RN$ fits into an exact sequence of the form
$$ 0  \to {}_RF \to {}_RN \to {}_RG \to 0$$
where ${}_RF$ is a free left $R$-module and ${}_RG$ is a free left $Q$-module.

\bigskip

 It is well known that every left module has an Enochs-cotorsion envelope.

\begin{theorem}
If $Q$  is a left perfect ring, then every left $R$-module has an 
$\mathcal{MC}$-envelope.
\end{theorem}

\begin{proof}
Let  $\mathcal{P}$ be the class of all projective left $Q$-modules, and consider $\mathcal{P}$ as a class of left $R$-modules.
By Lemma~\ref{ExSten}, this class of left $R$-modules is closed under extensions.
Clearly $\mathcal{P} ^\bot  =  {Q} ^\bot $, because if, for some left $R$-module $M$, $\Ext^1_R(Q,M)=0$, then $\Ext^1_R(\bigoplus Q,M)\cong\prod \Ext^1_R(Q,M)=0$. On the other
hand, $Q$ is left perfect, and so, by Bass' theorem, every direct limit of
projective left
$Q$-modules is projective. Thus the class $\mathcal{P} $ is closed under
direct limits.
Now assume that $M$ is a left $R$-module. By \cite [Theorem 6.11]{approx},
there exists a short exact sequence $0 \to M \to  P \to N \to 0$ where $ M \to P$ is a special $\mathcal{MC}$-preenvelope and
$P$ is 
is the union of a continuous chain of submodules, $\left\{\,P_ \alpha \mid  \alpha <  \lambda \,\right\}$
 such that $P_0 = M$ and $P_{\alpha +1 } / P_\alpha$ is isomorphic to a direct sum of copies of $Q$
 for each $\alpha < \lambda$. Since $N \cong P/M$, it follows that 
 $N$ is $\{Q\}$-filtered. Thus 
 $N$ is a free $Q$-module by Lemma~\ref{ExSten}, and so $N$ is an element of
$\mathcal{P}$. Therefore $M$ has an $\mathcal{MC}$-envelope \cite [Theorem 5.27]{approx}.
\end{proof}

A left $R$-module $_RM$ is called {\em weak-injective} if $\Ext^1_R(I, M) = 0$ for all modules $I$ of weak dimension $\leq 1$.

\begin{Lemma}\label{3.6} Weak-injective left $R$-modules are  $h$-divisible and Matlis-cotorsion.\end{Lemma}

\begin{Proof} Let $_RM$ be a weak-injective module. Since $_RQ$ is flat, we have that $_RK$ has weak dimension $\leq 1$.
So $\Ext^1_R(K, M) = 0$. From Theorem~\ref{2.2}(c), we get that $M/h(M)=0$ and $\Ext^1(Q,M)=0$, that is, $M$ is $h$-divisible and Matlis-cotorsion.\end{Proof}

Therefore the class $\mathcal{WI}$ of weak-injective modules is a subclass of $\mathcal{HD}$. We denote by $_R {\mathcal{P}_1  },{}_R{\mathcal{F}_1}$ and $ {}_R{\mathcal{ D}}$ the classes of all left $R$-modules of projective dimension $\le 1$, of weak dimension $\le 1$ and divisible, respectively.

\begin{proposition}\label{4.4} Under the hypotheses of this section, the following conditions hold:

{\rm (i)} $^\bot(_R \mathcal{HD}) \subseteq{}_R {\mathcal{P}_1  } $. 

{\rm (ii)} If $_R{\mathcal{F}_1}  ={} ^\bot (_R{\mathcal{ D}})$, then $\Fdim(_QQ)= 0$ and so $Q$ is left perfect.
\end{proposition}

\begin{proof}
(i) Assume that $M \in{} ^\bot (_R \mathcal{HD})$.
Let $_RN$ be a left $R$-module and $E$ be its injective hull.
Consider the exact sequence $0 \to N \to E \to E/N \to 0$. Since $_RE$ is injective, $\Ext^1_R(M,E/N)\cong \Ext^2_R(M,N)$.
Now $E/N$ is $h$-divisible and $M \in{} ^\bot  \mathcal{HD}$, so that $\Ext^1_R(M,E/N)=0$. Hence $\Ext^2 ({}_RM, {}_RN) = 0$.

(ii) It is enough to show that the $Q$-modules of projective dimension $\le 1$ are 
projective. Let $_QM$ be a $Q$-module of p.d.~$\le 1$. Then there exists an exact sequence $0 \to P_1 \to P_0 \to M \to 0$ 
with $P_1$ and $P_0$ projective $Q$-modules. Thus $_QP_i$ is a direct summand of $_QQ^{(X)}$, hence $_RP_i$ is a direct summand of the flat $R$-module $_RQ^{(X)}$.
So  $P_1$ and $P_0$ are flat $R$-modules, and hence $_RM$ is of weak dimension $\le1$. Since $_RM\in\Cal F_1={}^\bot(_R\Cal D)$ and $P_1\in{}_R\Cal H\Cal D\subseteq{}_R\Cal D$, 
the short exact sequence $0 \to P_1 \to P_0 \to M \to 0$ splits in $R\lMod$, and so in $Q\lMod$. Therefore $_QM$ is projective.  
\end{proof}

\section{$\Cal F$ is a 1-topology}\label{2}

As we have already said in (12), the Gabriel topology $\Cal F$ always has a basis consisting of finitely generated right ideals. Now we will suppose that the Gabriel topology $\Cal F$ is a 1-{\em topology}, that is, that $\Cal F$  has a basis consisting of principal right ideals \cite[Proposition~XI.6.1]{Stenstrom}. Thus $\Cal F$ is completely determined by the set $S:=\{\, s\in R\mid sR\in\Cal F\,\}$, which is a multiplicatively closed subset of $R$ satisfying: (1) If $a,b\in R$ and $ab\in S$, then $a\in S$. (2) If $s\in S$ and $a\in R$, then there are $t\in S$ and $b\in R$ such that $sb=at$ \cite[Proposition~VI.6.1]{Stenstrom}. Moreover, the elements of $S$ are not right zero-divisors in $R$, because $sR$ is a dense right ideal of $R$ for every $s\in S$, so $sR=(sR:1)$ has zero left annihilator  \cite[Proposition~VI.6.4]{Stenstrom}, and $s$ is not a right zero-divisor.

\medskip

For instance, consider the following trivial example. Suppose $Q=R$. There is not doubt that the identity $R\to Q$ is a bimorphism and that ${}_RR$ is a flat left $R$-module, so that $R$ is the least element in $\Cal L_R$. The corresponding multiplicatively closed subset  $S$ is then the set of all right invertible elements of $R$. The only torsion right $R$-module is the zero module. All right $R$-modules are torsion-free.

\medskip

Thus, in the rest of this section, {\em we will suppose that $R$ is a ring and $S$ is a multiplicatively closed subset of $R$ satisfying:} (1) {\em If $a,b\in R$ and $ab\in S$, then $a\in S$.} (2) {\em If $s\in S$ and $a\in R$, then there are $t\in S$ and $b\in R$ such that $sb=at$.} (3) {\em The elements of $S$ are not right zero-divisors.}

\begin{Lemma} Let $\Cal F$ be the Gabriel topology consisting of all right ideals $I$ of $R$ such that $I\cap S\ne 0$, let $R_{\Cal F}$ be the localization and $\varphi \colon R\to R_{\Cal F}$ be the canonical mapping. Then $\Cal F$ consists of dense right  ideals, $\varphi$ is a bimorphism  and ${}_RR_{\Cal F}$ is a flat left $R$-module.\end{Lemma}

\begin{Proof} In order to show that the right ideal $sR$ is dense for every $s\in S$, we must prove that, for every $s\in S$ and $a\in R$, the right ideal $(sR:a)$ has zero left annihilator \cite[Proposition~VI.6.4]{Stenstrom}. Now if  $s\in S$ and $a\in R$, then $(sR:a)\in\Cal F$ by  \cite[Property~T3 on Page 144]{Stenstrom}, so that $(sR:a)$ contains an element $t\in S$. Hence the right ideal $sR$ is dense because  the elements $t\in S$ is  not a right zero-divisor.
\end{Proof}

It follows that the torsion submodule of a right $R$-module $M_R$ is the set of all elements $x\in M_R$ for which there exists an element $s\in S$ with $xs=0$. In particular, a right $R$-module $M_R$ is torsion-free if right multiplication $\rho_s\colon M_R\to M_R$ by $s$ is an abelian group monomorphism for every $s\in S$. Dually, we will say that a right $R$-module $M_R$ is {\em divisible} if right multiplication $\rho_s\colon M_R\to M_R$ by $s$ is an abelian group epimorphism for every $s\in S$, that is, if $Ms=M$ for every $s\in S$. Every homomorphic image of a divisible right $R$-module is divisible. If $A$ is a submodule of a right $R$-module $B_R$ and if $A_R$ and $B/A$ are divisible, then $B$ is divisible. Any sum of divisible submodules is a divisible submodule, so that every right $R$-module $M_R$ contains a greatest divisible submodule, denoted by $d(M_R)$. A right $R$-module $M_R$ is {\em reduced} if $d(M_R)=0$. For every module $M_R$, $M_R/d(M_R)$ is reduced.

\begin{remark}{\em (a) It is very important to stress that all the concepts we have defined until now in Sections~\ref{1} and~\ref{2}, like divisible right $R$-module, reduced right $R$-module, $h$-divisible right or left $R$-module, and Matlis-cotorsion $R$-module are relative, in the sense that they depend on the fixed multiplicatively closed set $S$ (in Section~\ref{2}) or on the overring $Q$ of $R$ (in Section~\ref{1}). We have decided not to use a terminology  like 
$S$-divisible right $R$-module, $S$-reduced right $R$-module, $Q$-$h$-divisible right or left $R$-module, $Q$-Matlis-cotorsion $R$-module in order not to make the terminology itself too heavy.

(b) The localization $R_{\Cal F}$ is not the right ring of quotients $R[S^{-1}]$ of $R$ with respect to $S$ in general, as the following example shows.}\end{remark}

\begin{example} {\em Let $k$ be a division ring, $V_k$ an infinite dimensional right vector space over $k$ and 
$R:= \End(V_k)$. Since $R$ is Von Neumann regular, the only bimorphisms $R\to Q$ with $_RQ$ flat are isomorphisms \cite[Proposition~XI.1.4]{Stenstrom}. Thus, in this case, we have $\Cal {L}_R=\{R\}$, so that without loss of generality we can assume $Q=R$, $\Cal F=\{R\}$, $R_{\Cal F}=R$ and $S$ the set of all right invertible elements of $R$. The right invertible elements of $R$ are exactly the epimorphisms $V_k\to V_k$. Let us show that the right ring of quotients $R[S^{-1}]$ of $R$ with respect to $S$ does not exists. Suppose the contrary, and let $\varphi\colon R\to R[S^{-1}]$ denote the canonical morphism. Fix a direct-sum decomposition $V_k=U\oplus W$ with $U\cong W\cong V_k$. Then it is easy to construct epimorphisms $f,g\colon V_k\to V_k$ and monomorphisms $f',g'\colon V_k\to V_k$ with $ff'=1_V$, $gg'=1_V$, $fg'=0$, $gf'=0$ and $f'f+g'g=1_V$. As $f,g\in S$, it follows that $\varphi(f),\varphi(g)$ are invertible in $R[S^{-1}]$, with inverse $\varphi(f'),\varphi(g')$ respectively. Now $e:=f'f$ is an idempotent in $R$, with $1-e=g'g$. Thus $\varphi(f'f)=1$ in $R[S^{-1}]$, and similarly $\varphi(g'g)=1$ in $R[S^{-1}]$. It follows that $1=0$ in $R[S^{-1}]$, so that $R[S^{-1}]$ is the zero ring. Thus $\varphi(s)=0$ for every $s\in S$. Therefore $R$ is the zero ring as well, a contradiction. This proves that the localization $R_{\Cal F}=R$ is not the right ring of quotients $R[S^{-1}]$ of $R$ with respect to $S$.}\end{example}

\begin{proposition}\label{xyk} Suppose that the ring $Q$ is directly finite. Then:

(1) The elements of $S$ are regular elements of $R$, invertible in $Q$.

(2) The set $S$ is a right denominator set in $R$, and $Q$ is the right ring of quotients $R[S^{-1}]$ of $R$ with respect to $S$.
\end{proposition}

\begin{Proof} (1) If $s\in S$, then $sR\in\Cal F$, so $sQ=Q$. Thus $s$ is right invertible in~$Q$. But $Q$ is directly finite, so right invertible elements of $Q$ are invertible in $Q$. In particular, $s$ is regular in $R$.

(2) follows immediately from (1).
\end{Proof}

Proposition~\ref{xyk} will be later applied in particular to the case in which $Q$ is right (or left) perfect, hence semilocal, hence directly finite. 

\begin{corollary}\label{3.3} Suppose that the ring $Q$ is directly finite. Then for every torsion right  $R$-module $M_R$, the canonical mapping $$\pi\colon \Hom({}_RK_R, M_R)\otimes_RK \to h(M_R),$$ defined by $\pi(f\otimes x)=f(x)$ for every $f\in \Hom(K_R,M_R)$, is a right $R$-module isomorphism.\end{corollary}

\begin{Proof} We saw in Corollary~\ref{Zahra} that $\pi$ is surjective. As far as injectivity is concerned, notice that
every element of $\Hom({}_RK_R, M_R)\otimes_RK$ can be written in the form $$\sum_{i=1}^n f_i\otimes (q_i+R).$$ Now $q_i=r_is^{-1}$ for a suitable $s\in S$ \cite[Lemma~10.2(a)]{anintroductionGoodearl}, so that every element of $\Hom({}_RK_R, M_R)\otimes_RK$ can be written in the form $f\otimes (s^{-1}+R)$. Suppose that $f\otimes (s^{-1}+R)\in\ker\pi$, that is, $f(s^{-1}+R)=0$. Let $p \colon Q\to K$ denote the canonical projection, so that $f p \colon Q_R\to M_R$ is a morphism whose kernel contains $s^{-1}$. Compose this with left multiplication by $s^{-1}$
$$\lambda\colon Q_R\to Q_R,$$ which is an automorphism, getting a morphism  $f p \lambda\colon Q_R\to M_R$ whose kernel contains $1$. Thus $f p \lambda$ factors through a suitable morphism $
g\colon K_R\to M_R$, so that $f p \lambda=gp $. If $\lambda'\colon Q_R\to Q_R$ is left multiplication by $s$, then $f p =gp \lambda'$, that is, $f(x+R)=g(sx+R)$ for every $x\in Q$. This proves that $f=gs$. Then $f\otimes (s^{-1}+R)=gs\otimes (s^{-1}+R)=g\otimes 0=0$. Therefore $\pi$ is also injective.
\end{Proof}

Let $M_R$ be a right $R$-module. For every element $x\in M_R$, there is a right $R$-module morphism $R_R\to M_R$, $1\mapsto x$. Tensoring with $_RK$, we get a right $R$-module morphism $\lambda_x\colon K_R\to M\otimes_RK$, defined by $\lambda_x(k)=x\otimes k$.  The mapping $\lambda\colon M_R\to\Hom(K_R, M\otimes_RK)$, defined by $\lambda(x)=\lambda_x$ for every $x\in M_R$, is a right $R$-module morphism, as is easily checked. Here the right $R$-module structure on $\Hom(K_R, M\otimes_RK)$ is given by the multiplication defined, for every $f\in \Hom(K_R, M\otimes_RK)$ and $r\in R$, by $(fr)(k)=f(rk)$ for all $k\in K$.

\begin{theorem} Suppose $Q$ directly finite. Let $M_R$ be an $h$-reduced torsion-free right $R$-module. Then the canonical mapping $\lambda\colon M_R\to\Hom(K_R, M\otimes_RK)$ is injective and its cokernel is isomorphic to $\Ext_R^1(Q_R,M_R)$.\end{theorem}

\begin{Proof} The proof is organized in seven steps.

\smallskip

{\em Step 1: Every element of $M\otimes_RK$ can be written in the form $x\otimes (s^{-1}+R)$ for suitable $x\in M_R$ and $s\in S$.}

Any element of $M\otimes_RK$ is of  the form $\sum_{i=1}^nx_i\otimes (r_is_i^{-1}+R)$. Reducing to the same denominator \cite[Lemma~4.21]{anintroductionGoodearl}, we find elements $r'_i\in R$ and $s\in S$ such that $s_ir'_i=s$ for every $i$. Multiplying by $s^{-1}$ on the right and by $s_i^{-1}$ on the left, we get that $r'_is^{-1}=s_i^{-1}$. Thus $\sum_{i=1}^nx_i\otimes (r_is_i^{-1}+R)=\sum_{i=1}^nx_i\otimes (r_ir'_is^{-1}+R)=\left(\sum_{i=1}^nx_ir_ir'_i\right)\otimes (s^{-1}+R)$ is of the form $x\otimes (s^{-1}+R)$.

\smallskip

{\em Step 2: Let $s$ be an element of $S$. The elements $y$ of $M\otimes_RK$ such that $ys=0$ are those that can be written in the form $x\otimes (s^{-1}+R)$ for a suitable $x\in M_R$.}

Clearly, $\left(x\otimes (s^{-1}+R)\right)s=0$. Conversely, let $y$ be an element of $M\otimes_RK$ such that $ys=0$. By Step 1, we have that $y=z\otimes (t^{-1}+R)$ for suitable elements $z\in M_R$ and $t\in S$. Taking the same denominator again, we get $a,b\in R$ and $u\in S$ with $sa=u$ and $tb=u$, so that $au^{-1}=s^{-1}$ and $bu^{-1}=t^{-1}$ in $Q$. Hence $y=z\otimes (t^{-1}+R)=zb\otimes (u^{-1}+R)$. From the short exact sequence (\ref{a}) in Theorem~\ref{2.2}, we see that the condition $ys=0$ implies that \begin{equation}zb\otimes (u^{-1}s)=x\otimes 1\label{**}\end{equation} in $M\otimes_RQ$ for some $x\in M_R$. Now $au^{-1}=s^{-1}$, so $au^{-1}s=1$. As $Q$ is directly finite, one-sided inverses are two-sided inverses, hence 
$u^{-1}sa=1$. Thus, multiplying (\ref{**}) by $a$ on the right, we get that $zb\otimes 1=x\otimes a$, from which $zb-xa=0$. Thus $y=zb\otimes (u^{-1}+R)=xa\otimes (u^{-1}+R)=x\otimes (au^{-1}+R)=x\otimes (s^{-1}+R)$, as desired.

\smallskip

{\em Step 3: If $x\in M_R$, $r\in R$ and $s\in S$, then $x\otimes(rs^{-1}+R)=0$ in $M\otimes_RK$ if and only if $xr\in Ms$.}

From the short exact sequence (d) in Theorem~\ref{2.2}, we see that $x\otimes(rs^{-1}+R)=0$ in $M\otimes_RK$ if and only if there exists $y\in M_R$ such that $x\otimes(rs^{-1})=y\otimes 1$ in $M\otimes_RQ$, if and only if $x\otimes r=y\otimes s$, if and only if $xr-ys=0$. That is, $x\otimes(rs^{-1}+R)=0$ in $M\otimes_RK$ if and only if there exists $y\in M_R$ with $xr=ys$, that is, if and only if $xr\in Ms$.

\smallskip

{\em Step 4: $\lambda$ is injective.}

The submodule $\ker\lambda$ of $M_R$ is torsion-free because it is a submodule of the torsion-free module $M_R$. Let us show that $\ker\lambda$ is also divisible. Let $x$ be an element of $\ker\lambda$ and $s\in S$. Then $\lambda_x=0$, so that $x\otimes k=0$ for every $k\in K$. It follows that $x\otimes(rt^{-1}+R)=0$ in $M\otimes_RK$ for every $r\in R$ and $t\in S$. By Step 3, $xr\in Mt$ for every $r\in R$ and $t\in S$. In particular, $x\in Ms$, so that $x=ys$ for some $y\in M_R$. In order to conclude, it suffices to show that $y\in\ker\lambda$, that is, that $\lambda_y=0$, equivalently that $y\otimes K=0$ in $M\otimes_RK$. But $y\otimes K=y\otimes sK=ys\otimes K=x\otimes K=0$. Thus $y\in\ker\lambda$. This proves that $\lambda_y=0$, so that $\ker\lambda$ is a divisible submodule of $M_R$. As $\ker\lambda$ is both torsion-free and divisible, right multiplication by any element $s\in S$ is an automorphism of the abelian group $\ker\lambda$. Thus $\ker\lambda$ has a unique right $Q$-module structure that extends the right $R$-module structure. In particular, $\ker\lambda$ is $h$-divisible. But $M_R$ is $h$-reduced, so that $\ker\lambda=0$.

\smallskip

Thus we have a short exact sequence \begin{equation}0\rightarrow
 M_R \stackrel{\lambda}{\longrightarrow} \Hom(K_R, M\otimes_RK)\rightarrow C_R\rightarrow 0,\label{e}\tag{e}\end{equation} where $C_R$ denotes the cokernel of $\lambda$. 
 
 \smallskip

 {\em Step 5: $C_R$ is torsion-free.}
 
 Suppose $f\in  \Hom(K_R, M\otimes_RK)$, $s\in S$ and $fs\in\lambda(M_R)$. We must prove that $f\in\lambda(M_R)$. Now $fs\in\lambda(M_R)$ implies that there exists $x\in M_R$ with $fs=\lambda_x$, that is $f(sk)=x\otimes k$ for every $k\in K$. In particular, $0=f(1_Q+R)=f(s(s^{-1}+R))=x\otimes(s^{-1}+R)$. By Step 3, we get that $x\in Ms$. Hence there exists $y\in M_R$ with $x=ys$. It follows that $f(sk)=x\otimes k=ys\otimes k=y\otimes sk$. As $_RK$ is divisible, we get that $f(k)=y\otimes k$ for every $k\in K$, i.e., $f=\lambda_y\in \lambda(M_R)$, as desired.
 
  \smallskip

 {\em Step 6: $C_R$ is divisible.}
 
Assume that $f\in  \Hom(K_R, M\otimes_RK)$ and $s\in S$. We must prove that there exist $g\in  \Hom(K_R, M\otimes_RK)$ and $x\in M_R$ such that $f=gs+\lambda_x$. Now $f\in  \Hom(K_R, M\otimes_RK)$ and $s\in S$ imply that $f(s^{-1}+R)$ is an element of $M\otimes_RK$ such that $(f(s^{-1}+R))s=0$. By Step 2, $f(s^{-1}+R)=x\otimes (s^{-1}+R)$ for a suitable $x\in M_R$. Thus $(f-\lambda_x)(s^{-1}+R)=0$. It follows that if $\pi\colon K_R=Q/R\to Q/s^{-1}R$ denotes the canonical projection, there exists a morphism $\overline{g}\colon Q/s^{-1}R\to M\otimes_RK$ such that $f-\lambda_x=\overline{g}\pi$. Similarly, if $\ell_s\colon K_R\to K_R$ denotes the right $R$-module morphism $\ell_s\colon k\mapsto sk$, there exists an isomorphism $\overline{\ell_s}\colon Q/s^{-1}R\to Q/R=K_R$ such that $\ell_s=\overline{\ell_s}\pi$. Set $g:=\overline g\circ(\overline{\ell_s})^{-1}$, so that $g\colon K_R\to M\otimes_RK$. Then $gs+\lambda_x=g\circ\ell_s+\lambda_x=\overline g\circ(\overline{\ell_s})^{-1}\circ\overline{\ell_s}\circ\pi+\lambda_x=\overline g\circ\pi+\lambda_x=f$, as desired.
 
 \smallskip

 {\em Step 7: $C_R\cong \Ext_R^1(Q_R,M_R)$.}
 
 Apply the functor $\Hom(Q_R,-)$ to the short exact sequence (\ref{e}), getting an exact sequence \begin{equation}\begin{array}{l}\Hom(Q_R, \Hom(K_R, M\otimes_RK))\rightarrow \Hom(Q_R, C_R)\rightarrow \\ \qquad\qquad\rightarrow\Ext^1_R(Q_R, M_R) \rightarrow \Ext^1_R(Q_R, \Hom(K_R, M\otimes_RK)).\end{array}\label{f}\tag{f}\end{equation} The right $R$-module $ \Hom(K_R, M\otimes_RK)$ is Matlis-cotorsion and $h$-reduced by Theorem~\ref{1.6}, so that the first and the last module in the exact sequence (\ref{f}) are zero. It follows that $\Hom(Q_R, C_R)\cong \Ext^1_R(Q_R, M_R) $. Now $C_R$ is torsion-free and divisible (Steps 6 and 7), so that right multiplication by any element of $S$ is an automorphism of the abelian group $C$. It follows that $C$ has a unique right $Q$-module structure that extends the right $R$-module structure on $C_R$. In particular $C_Q\cong \Hom(Q_R, C_R)$ by (7). It follows that $C_R\cong \Hom(Q_R, C_R)\cong \Ext^1_R(Q_R, M_R) $, which concludes the proof of the Theorem.
\end{Proof}

\begin{corollary}\label{3.5} Suppose $Q$ directly finite. Let $M_R$ be an $h$-reduced torsion-free Matlis-cotorsion right $R$-module. Then the canonical mapping $$\lambda\colon M_R\to\Hom(K_R, M\otimes_RK)$$ is an isomorphism.\end{corollary}

Thus we have generalized to our setting the Matlis category equivalence \cite[Corollary 2.4]{22}:

\begin{theorem}\label{main1} Suppose $Q$ directly finite. Then there is an equivalence of the category $\Cal C$ of $h$-reduced torsion-free
Matlis-cotorsion right $R$-modules with the category $\Cal T$ of $h$-divisible torsion right $R$-modules, given by $$-\otimes_RK\colon\Cal C\to\Cal T\qquad\mbox{\rm and}\qquad \Hom(K_R,-)\colon\Cal T\to~\Cal C.$$\end{theorem}

\begin{Proof} Corollaries~\ref{3.3} and~\ref{3.5}.\end{Proof}

\section{Left and right flat bimorphisms, and $1$-topologies}\label{4}

In this section, {\em $R$ and $Q$ are rings, $\varphi \colon R\to Q$ is a bimorphism in the category of rings, and the $R$-$R$-bimodule ${}_RQ_R$ is a flat both as a left $R$-module and as a right $R$-module.} Therefore $\varphi\colon R\to Q$ is the canonical homomorphism of $R$ both into its right localization $R_{\Cal F}$, where $\Cal F$ is the right Gabriel topology $\Cal \{\,I\mid I$ is a right ideal of $R$ and $\varphi(I)Q=Q\,\}$
and into its
 left localization $R_{\Cal G}$, where $\Cal G=\{\,J\mid J$ is a left ideal of $R$ and $Q\varphi(J)=Q\,\}$ is a Gabriel topology consisting of dense left ideals and with a basis consisting of finitely generated left ideals. In order to apply the results of Section~\ref{2}, we will {\em suppose that $\Cal F$ and $\Cal G$ are 1-topologies and that $\Fdim(Q_Q) = 0$.} In particular, $Q$ is right perfect, and so directly finite. 
Correspondingly to $\Cal F$ and $\Cal G$, we have the two sets $S:=\{\, s\in R\mid sQ=Q\,\}$ and $T:=\{\, t\in R\mid Qt=Q\,\}$, so that $S$ (resp.~$T$) consists of all the elements of $R$ that are right invertible (resp.~left invertible) in $Q$. But $Q$ is directly finite, which implies that $S=T$ consists of regular elements of $R$ and $Q=R[S^{-1}]=[S^{-1}]R$. Moreover, as the elements of $S$ are regular, the ring $Q$ is contained in the classical right ring of quotients $R[S^{-1}_{\reg}]$ of $R$, where $S_{\reg}$, denotes the set of all regular elements of $R$ \cite[p.~52]{Stenstrom}. Since every element $s\in S_{\reg}$ is invertible in the directly finite ring $Q$, it follows that $Q=R[S^{-1}_{\reg}]$. Similarly, $Q=[S^{-1}_{\reg}]R$, and $S=S_{\reg}$.
 
 Thus the situation now is the following. We have, in this section, {\em a ring $R$ for which the set $S$ of all its regular elements is both a right denominator set and a left denominator set (right and left Ore ring), and we assume that $Q=R[S^{-1}]=[S^{-1}]R$ (the classical right and left ring of  quotients of $R$) and that $\Fdim(Q_Q)=0$.}
 
 \begin{remark}\label{right, but not left}{\rm Notice that if $R$ is a right Ore domain that is not left Ore, then the right field of quotients $Q$ of $R$ is flat as a left $R$ module, but is not flat as a right $R$-module \cite[paragraph after the proof of Proposition~0.8.6]{CohnFIRandLoc}. Thus the results of the previous section apply to this extension $Q$ of $R$, but the results in this section do not.}\end{remark}
 
We are now finally ready to prove, in the noncommutative case, the result, due to Fuchs and Salce in the commutative case, which we mentioned in the Introduction. In  \cite{FS}, Fuchs and Salce proved the equivalence of the nine equivalent conditions listed in the Introduction for modules over commutative rings $R$ with perfect quotient ring $Q$. Now we prove that the equivalence of seven of their conditions also holds for noncommutative right and left Ore rings $R$ for which  $\Fdim(Q_Q) = 0$. Here
  $Q = R[S^{-1}] = [S^{-1}] R$, where $S$ is the set of all regular elements of $R$. Notice that a commutative ring $Q$ is perfect if and only if $\Fdim(Q_Q) = 0$ \cite[pp.~466--468]{Bass}. 
  
 \begin{theorem}\label{7.1} Assume that $R$ is a right and left Ore ring and  $\Fdim(Q_Q) = 0$, where
  $Q = R[S^{-1}] = [S^{-1}] R$. Then 
the following conditions are equivalent:
 
(i) Flat right
$R$-modules are strongly flat.
 
 (ii) Matlis-cotorsion right $R$-modules are Enochs-cotorsion.

(iii) $h$-divisible right $R$-modules are weak-injective.
 
(iv) Homomorphic images of weak-injective right $R$-modules are weak-injective.
  
(v) Homomorphic images of injective right $R$-modules are weak-injective.

(vi) Right $R$-modules of $\wed \le 1$ are of $\pd\le1$.
 
(vii) The cotorsion pairs $(\Cal P_1,\Cal D)$ and $(\Cal F_1,\Cal W\Cal I)$ coincide.
 
 (viii) Divisible right $R$-modules are weak-injective.
 
\end{theorem}

\begin{proof}
$(i) {}\Leftrightarrow{} (ii)$ is clear, because 
$({}_R\mathcal{SF} , {}_R{\mathcal{MC}})$ and  $ ({}_R\mathcal{F},  {}_R\mathcal{EC})$
are cotorsion pairs, ${}_R\mathcal{SF} \subseteq {}_R\mathcal{F}$ and  ${}_R{\mathcal{MC}}\supseteq {}_R{\mathcal{EC}}$. 

$(ii) {}\Rightarrow{} (iii)$ Let $D_R$ be an $h$-divisible module. By sequence (\ref{b}) of Theorem~\ref{2.2}, we have an exact sequence of right $R$-modules 
$$ 0 \to\Hom(K, D) \to\Hom(Q, D) \to  D \to 0$$
 Let $M_R$ be a module of weak dimension $\le1$. In order to prove $(iii)$, we must show that $\Ext^1_R(M, D) = 0.$ 
We have the exact sequence  \begin{equation}\Ext^1_R(M,\Hom(Q, D)) \to \Ext^1_R(M, D) \to \Ext^2(M,\Hom(K, D)).\label{ooo}\end{equation} Now 
$\Hom(Q, D)$ is a right $Q$-module and so $\Ext^1_R(M,\Hom(Q, D)) \cong\Ext^1_Q(M\otimes Q,\Hom(Q, D))$ by (18). The module $M_R$ has weak dimension $\le1$, and $_RQ$ is flat, so that the $Q$-module $M\otimes_RQ$ has weak dimension $\le1$. But $Q$ is perfect, so that the $Q$-module $M\otimes_RQ$ has projective dimension $\le1$. 
Since $\Fdim(Q_Q) = 0$, $M \otimes Q$ is projective, and so $\Ext^1_Q(M\otimes Q,\Hom(Q, D)) = 0$. By (18), the first $\Ext$ in the sequence (\ref{ooo}) is zero. 
On  the other hand, $M_R$ has weak dimension $\le1$, so that there exists an exact sequence $0 \to N_R \to P_R \to M_R \to 0 $, where $N_R$ is flat and $P_R$ is projective.  Applying to this exact sequence the functor $\Hom(-,\Hom(K, D))$, we get an exact sequence $\Ext^1_R(N,\Hom(K, D))\to \Ext^2_R(M,\Hom(K, D))\to \Ext^2_R(P,\Hom(K, D))$. The last module is zero because $P$ is projective, and the first module is also zero because
$\Hom(K, D)$ is Matlis-cotorsion by Theorem~\ref{1.6}, and so  Enochs-cotorsion by $(ii)$. This implies that 
 $\Ext^2_R(M,\Hom(K, D)) = 0$. From the exact sequence (\ref{ooo}), we get that $\Ext^1_R(M, D) = 0$, as desired.
 
$(iii) {}\Rightarrow{} (iv)$ follows from Lemma~\ref{3.6}, and $(iv) {}\Rightarrow{} (v)$ is trivial. 

$(v) {}\Rightarrow{} (vi)$. Let $M$ be a right $R$-module and $A$ be a right $R$-module of weak  dimension $\le1$.
We want to show that $\Ext^2(A, M) = 0$. Let $E$ be the injective hull of $M$ and consider the exact sequence 
$0 \to M \to E \to E/M \to 0$. We get the exact sequence 
$\Ext^1 (A, E/M) \to\Ext^2 (A, M) \to\Ext^2(A, E) $, where the first $\Ext^1$ is zero, because $E/M$ is weak-injective by $(v)$, and the last $\Ext^2$ is zero because $E$ is injective.  So $A$ is of projective dimension $\le1$. 
 
$(vi) {}\Rightarrow{} (vii)$.
First of all we show that $({\Cal P}_1, \mathcal{HD})$ is a cotorsion pair. 
Let $M$ be a right $Q$-module  and $P_R\in   {\Cal P}_1$. As the module $P_R$ has projective dimension $\le1$, and $_RQ$ is flat, the $Q$-module $P\otimes_RQ$ has projective dimension $\le1$. But $\Fdim(Q_Q) = 0$, so $P \otimes Q$ is projective, and thus, from (18), we get that $\Ext^1_R(P,M) \cong  \Ext^1_Q(P \otimes Q,M) = 0$. This shows that $\Ext^1_R(P,M) =0$ for every right $Q$-module $M$ and every $P_R\in   {\Cal P}_1$. If $N$ is an $h$-divisible $R$-module, there is an exact sequence $0\to K\to M\to N\to 0$ for some $Q$-module $M$. From this sequence, we get the exact sequence $\Ext^1(P,M)\to \Ext^1(P,N)\to \Ext^2(P,K)$. The first module is zero because $M$ is a $Q$-module, and the last module is zero because $P$ is in $\Cal P_1$. We have thus proved that $\Ext^1_R(P,N) = 0$ for every module $P$ in $\Cal P_1$ and every $h$-divisible $R$-module $N$. This proves that  $\Cal H\Cal D \subseteq {\Cal P}_1 ^\bot $ and 
${\Cal P}_1 \subseteq{} ^\bot \Cal H \Cal D$. We also know that $^\bot \Cal H\Cal D \subseteq {\Cal P}_1$ (Proposition~\ref{4.4}(i)) and that 
 ${\Cal P}_1^ \bot ={\Cal F}_1^ \bot = \mathcal{WI} \subseteq \mathcal{HD}$ (by $(vi)$ and Lemma~\ref{3.6}). Therefore 
$({\Cal P}_1, \mathcal{HD})$ is a cotorsion pair. 
 Now p.dim $(K_R) \leq 1$ by $(vi)$, so that the functor $\Ext^2(K,-)$ is zero. From the exact sequence of bimodules $ 0\rightarrow
R\rightarrow Q\rightarrow K\rightarrow 0$,  we get that $\Ext^2(Q,-)\cong \Ext^2(K,-)$, and so p.dim$(Q_R) \le 1$.
Therefore  $\mathcal{HD}_R = \mathcal{D}_R $ by  \cite[Corollary 4.14]{AS}.

$(vi) {}\Rightarrow{} (ii)$ Firstly, we will show that $Q^\bot$ is closed under homomorphic images. Let $M \in Q^\bot$ and $N$ be a submodule of $M$.
From the exact sequence $0 \to N \to M \to M/N \to 0$, we get the exact sequence $\Ext^1(Q, M) \to \Ext^1(Q, M/N) \to \Ext^2(Q, N)$.
The first $\Ext$ is zero because $M \in Q^\bot$, and the third $\Ext$ is also zero, because $Q$ is of projective dimension $\le1$ by $(vi)$. So $M/N \in Q^\bot$.
In order to prove $(ii)$, we must show that $\Ext^1_R(F, C) = 0$ for every $C \in Q^\bot $ and every flat right $R$-module $F$.
For any $h$-divisible module $H$, we have that $H\in\Cal D$, so that $H\in \Cal W\Cal I$ by $(vi){}\Rightarrow{}(vii)$.
Also, $F$, which is flat, belongs to $\Cal F_1$. Therefore $\Ext^1_R(F, H) = 0$. Thus we can assume that $C$ is not $h$-divisible. Consider the exact sequence
$0 \to h(C) \to C \to C/h(C) \to 0 $. We have the short exact sequence $\Ext^1(F, h(C)) \to \Ext^1(F, C) \to \Ext^1(F, C/h(C))$.
The first $\Ext$ is zero as we have just seen, and $C/h(C) \in Q^\bot$ because $Q^\bot$ is closed under homomorphic images. We want to show that $\Ext^1(F, C/h(C))=0$. Apply \cite[Theorem~3.5]{AS} to the injective ring epimorphism $R\to Q$. As we have already seen, the projective dimension of $Q$ is  $\le1$, so that condition (1) in \cite[Theorem~3.5]{AS} holds. Thus condition (4) holds, that is, the class $K^\bot $ is the class of modules generated by $Q$, that is, the class of $h$-divisible modules. Thus
the class $K^\bot $ is closed under extensions. Hence we can apply Proposition~\ref{factor}, and get that
$C/h(C)$ is $h$-reduced. So it suffices to assume that $C$ is $h$-reduced.
From the short exact sequence $0 \to R \to Q \to K \to 0$, we obtain the exact sequence
$0 \to F \to F\otimes Q \to F \otimes K \to 0$. Thus we have the exact sequence
$\Ext^1_R(F\otimes Q, C) \to \Ext^1_R(F, C) \to \Ext^2_R(F \otimes K, C)$.
The first $\Ext^1_R$ is zero: 
$\Ext^1_R (F\otimes Q, C) \cong \Ext^1_R (F, \Hom (Q, C))$ by \cite [Lemma 2.3] {FL}, and $\Hom (Q, C)=0$ because $C$ is $h$-reduced. Finally, let us prove that the third $\Ext^2_R$ is also zero. We have the short exact sequence of right $R$-modules $0\to F\otimes R\to F\otimes Q\to F\otimes K\to 0$. The right $R$-module $F\otimes Q$ is flat, and therefore $F\otimes K$ is of weak dimension $\le1$. By $(vi)$, 
$F \otimes K$ is of projective dimension $\le1$. Thus $\Ext^2_R(F \otimes K, C)=0$.

$(vii) {}\Rightarrow{} (viii)$ and $(viii) {}\Rightarrow{} (iii)$   are obvious. 
\end{proof}

\end{document}